\documentclass[twoside,11pt]{amsart}
\usepackage[a4paper,margin=2.5cm]{geometry}
\usepackage{microtype}
\usepackage{amsmath,amssymb,amsthm,enumerate,graphics,float,array,url,color,tikz,longtable,hhline,mathrsfs,latexsym}
\usepackage[all]{xy}
\usepackage[shortcuts]{extdash}

\theoremstyle{plain}
\newtheorem{theorem}{Theorem}
\newtheorem{thm}{Theorem}[section]
\newtheorem{lemma}[thm]{Lemma}
\newtheorem{cor}[thm]{Corollary}
\newtheorem{prop}[thm]{Proposition}
\theoremstyle{definition}
\newtheorem{question}[]{Question} 
\newtheorem{rem}[thm]{Remark}
\newtheorem{ex}[thm]{Example}

\usepackage[shortlabels]{enumitem}
\setlist[enumerate]{label=\rm{(\roman*)}}

\numberwithin{equation}{section}

\renewcommand{\leq}{\leqslant}

\renewcommand{\geq}{\geqslant}

\newcommand{\<}{\langle}
\renewcommand{\>}{\rangle}
\renewcommand{\epsilon}{\varepsilon}
\renewcommand{\emptyset}{\varnothing}
\DeclareMathOperator{\frat}{Frat}

\DeclareMathOperator{\diam}{diam}

\begin{document}

\author{Scott Harper}
\address{S. Harper, School of Mathematics, University of Bristol, Bristol BS8 1UG, UK, and Heilbronn Institute for Mathematical Research, Bristol, UK.}
\email{scott.harper@bristol.ac.uk}

\author{Andrea Lucchini}
\address{A. Lucchini, Dipartimento di Matematica ``Tullio Levi-Civita'', Universit\`a degli Studi di Padova, 35121 Padova, Italy}
\email{lucchini@math.unipd.it}
\thanks{The first author is grateful for the support of the London Mathematical Society through an Early Career Fellowship and for the hospitality of Universit\`a degli Studi di Padova. Both authors thank an anonymous referee for their helpful comments.}

\title[Generating graph of nilpotent groups]{Connectivity of generating graphs \\ of nilpotent groups}

\date{\today}

\begin{abstract}
Let $G$ be $2$-generated group. The generating graph of $\Gamma(G)$ is the graph whose vertices are the elements of $G$ and where two vertices $g$ and $h$ are adjacent if $G=\<g,h\>$. This graph encodes the combinatorial structure of the distribution of generating pairs across $G$. In this paper we study several natural graph theoretic properties related to the connectedness of $\Gamma(G)$ in the case where $G$ is a finite nilpotent group. For example, we prove that if $G$ is nilpotent, then the graph obtained from $\Gamma(G)$ by removing its isolated vertices is maximally connected and, if $|G| \geq 3$, also Hamiltonian. We pose several questions.
\end{abstract}

\maketitle

\section{Introduction} \label{s:intro}

Since the earliest days of group theory, generating sets for groups have led to many interesting, and often surprising, results. In recent years, the generating graph has provided a combinatorial framework for studying group generation and many new structural results have emerged. 

Let $G$ be a finite group. The \emph{generating graph} of $G$ is the graph $\Gamma(G)$ whose vertices are the elements of $G$ and where $g,h \in G$ are adjacent if $G=\<g,h\>$ (we do not include loops when $G$ is cyclic). Several strong structural results about $\Gamma(G)$ are known in the case where $G$ is simple, and this reflects the rich group theoretic structure of these groups. For example, if $G$ is a nonabelian simple group, then the only isolated vertex of $\Gamma(G)$ is the identity \cite{gk} and the graph $\Delta(G)$ obtained by removing the isolated vertex is connected with diameter two \cite{bgk} and, if $|G|$ is sufficiently large, admits a Hamiltonian cycle \cite{bglm} (it is conjectured that the condition on $|G|$ can be removed). Moreover, there has been much recent interest in attempting to classify the groups $G$ for which $\Gamma(G)$ shares the strong properties of the generating graphs of simple groups; all proper quotients of such groups are necessarily cyclic, so these groups are closely related to simple groups (see \cite[Conjecture~1.8]{bgk}, \cite[Conjecture~1.6]{bglm} and \cite{bg,h} for recent work in this direction). 

In this paper, we focus on groups at the other end of the spectrum: we establish structural results about the generating graphs of nilpotent groups. We emphasise that even for this class of groups the theory is intricate and this leads to several natural questions that we present.

Write $\Delta(G)$ for the graph obtained by removing the isolated vertices from $\Gamma(G)$. If $G$ is a $2$-generated soluble group, then the main theorem of \cite{cl1} states that $\Delta(G)$ is connected. Our first two theorems give significantly stronger versions of this result for nilpotent groups. 

The \emph{(vertex) connectivity} $\kappa(\Gamma)$ is the least number of vertices of $\Gamma$ that can be removed such that induced subgraph on the remaining vertices is disconnected (we say $\kappa(K_n)=n-1$). Since $\kappa(\Gamma)$ is at most the minimal vertex degree $\delta(\Gamma)$, we say $\Gamma$ is \emph{maximally connected} if $\kappa(\Gamma) = \delta(\Gamma)$.

\begin{theorem} \label{thm:conn}
Let $G$ be a finite $2$-generated nilpotent group. Then $\Delta(G)$ is maximally connected.
\end{theorem}

It is an open question whether $\Delta(G)$ is connected for every finite $2$-generated group, and we ask whether the following even stronger property holds.

\begin{question}\label{q:conn}
Is $\Delta(G)$ maximally connected for every $2$-generated finite group $G$?
\end{question}

We now present our second main theorem.

\begin{theorem} \label{thm:cycles}
Let $G$ be a nontrivial finite $2$-generated nilpotent group. Then
\begin{enumerate}
\item $\Delta(G)$ is Eulerian if and only if $G$ is not a cyclic group of even order
\item $\Delta(G)$ is Hamiltonian if and only if $G$ is not the cyclic group of order two.
\end{enumerate}
\end{theorem}

In \cite[Conjecture~1.6]{bglm}, it is conjectured that if $\Gamma(G)$ has at most one isolated vertex (which is necessarily the identity), then $\Delta(G)$ is Hamiltonian, and the authors of \cite{bglm} proved this conjecture for soluble groups and sufficiently large simple groups. Motivated by Theorem~\ref{thm:cycles} we propose the following question.

\begin{question} \label{q:hamiltonian} 
Is $\Delta(G)$ Hamiltonian for every $2$-generated finite group other than $C_2$?
\end{question}

We now study the total domination number of $\Delta(G)$, which has been the recent focus of attention in the case where $G$ is simple \cite{bh1,bh2}. Recall that the total domination number $\gamma_t(\Gamma)$ of a finite graph $\Gamma$ is the least size of a set $S$ of vertices of $\Gamma$ such that every vertex of $\Gamma$ is adjacent to a vertex in $S$. 

\begin{theorem} \label{thm:tdn}
Let $G$ be a finite $2$-generated nilpotent group. 
\begin{enumerate}
\item We have that $\gamma_t(\Delta(G))=1$ if and only if $G$ is cyclic.
\item Assume that $G$ has exactly $s \geq 1$ noncyclic Sylow subgroups and let $p$ be the smallest prime such that the Sylow $p$-subgroup of $G$ is not cyclic. Then $\gamma_t(\Delta(G)) \geq s+1$ with equality if $p \geq s$.
\end{enumerate}
\end{theorem}

We give more information about the total domination number in Section~\ref{s:tdn}, but the following remains open.

\begin{question} \label{q:tdn}
What is $\gamma_t(\Delta(G))$ for a general finite $2$-generated nilpotent group? 
\end{question}

Finally, we study the clique number $\omega$ and the chromatic number $\chi$. Here, the result follows from work of Mar\'oti and the second author, except in the case where $G$ is cyclic. In the statement of this result, we write $\phi$ for the Euler totient function and we write $\pi(n)$ for the number of distinct prime divisors of $n$.

\begin{theorem} \label{thm:clique_chromatic}
Let $G$ be a finite $2$-generated nilpotent group. Then $\omega(\Gamma(G)) = \chi(\Gamma(G))$. Moreover, 
\begin{enumerate}
\item if $G$ is not cyclic, then $\omega(\Gamma(G)) = \chi(\Gamma(G)) = p+1$ where $p$ is the smallest prime such that the Sylow $p$-subgroup of $G$ is not cyclic
\item if $G$ is cyclic of order $n$, then $\omega(\Gamma(G)) = \chi(\Gamma(G)) = \phi(n) + \pi(n)$.
\end{enumerate}
\end{theorem}

Theorem~\ref{thm:clique_chromatic} may lead the reader to ask whether the equality $\omega(\Gamma(G))=\chi(\Gamma(G))$ holds for an arbitrary $2$-generated finite group. It does not: by \cite{LM}, there are infinitely many nonabelian finite simple groups $G$ with $\omega(\Gamma(G)) < \chi(\Gamma(G))$. However the following question is open.

\begin{question}\label{q:clique_chromatic}
Does there exist a $2$-generated finite soluble group $G$ with $\omega(\Gamma(G)) < \chi(\Gamma(G))$?
\end{question}

\section{Preliminaries} \label{s:prelims}

Our graph theoretic notation is standard. In particular, for graphs $\Gamma$ and $\Delta$ we write
\begin{enumerate}[ ]
\item $V(\Gamma)$ and $E(\Gamma)$ for the vertex and edge sets of $\Gamma$ \\[-9pt]
\item $\delta_\Gamma(v)$ for the degree of a vertex $v \in V(\Gamma)$ and $\delta(\Gamma)$ for $\min_{v \in V(\Gamma)} \delta_\Gamma(v)$ \\[-9pt]
\item $\overline{\Gamma}$ for the complement of $\Gamma$ \\[-9pt]
\item $K_n$ and $\overline{K}_n$ for the complete and null graphs with $n$ vertices \\[-9pt]
\item $\Gamma \times \Delta$ for the \emph{direct} (or \emph{categorical} or \emph{tensor}) \emph{product}, whose vertex set is $V(\Gamma) \times V(\Delta)$ and where adjacency is defined as $(\gamma_1,\delta_1) \sim (\gamma_2,\delta_2)$ if $\gamma_1 \sim \gamma_2$ in $\Gamma$ and $\delta_1 \sim \delta_2$ in $\Delta$ \\[-9pt]
\item $\Gamma[\Delta]$ for the \emph{lexicographical product}, whose vertex set is $V(\Gamma) \times V(\Delta)$ and where adjacency is defined as $(\gamma_1,\delta_1) \sim (\gamma_2,\delta_2)$ if $\gamma_1 \sim \gamma_2$, or $\gamma_1 = \gamma_2$ and $\delta_1 \sim \delta_2$.
\end{enumerate}

\subsection{General groups}

For this section, let $G$ be a $2$-generated finite group. We begin with a straightforward observation. Recall that $\Delta(G)$ is the graph obtained from the generating graph $\Gamma(G)$ by removing the isolated vertices. 

\begin{lemma}\label{lem:complete}
Let $G$ be a $2$-generated finite group. Then $\Delta(G)$ is complete if and only if either $G \cong C_p$ for a prime $p$ or $G \cong C_2^2$.
\end{lemma}

\begin{proof}
Let $g \in V(\Delta(G))$. Clearly $g^{-1}\in \Delta(G)$, so either $G=\langle g, g^{-1}\rangle =\langle g \rangle$ or $g=g^{-1}$. It is easy to see that $\Delta(C_n)$ is complete if and only if $n$ is prime. Now assume that $G$ is not cyclic. Then every element in $V(\Delta(G))$ is an involution. This implies that $G$ is a dihedral group of order $2n$, for some $n \geq 2$, since it is generated by two involutions. Since $D_{2n} = \<a,b\>$ where $|a|=n$ and $|b|=2$, we deduce that $n=2$, so $G \cong C_2^2$. This completes the proof.
\end{proof}

Notice that $G = \<g,h\>$ if and only if $G/\frat(G) = \<g\frat(G), h\frat(G)\>$. Therefore, if $G$ is noncyclic, then we have the lexicographical product
\begin{equation} \label{eq:graph_decomp}
\Delta(G) = \Delta(G/\frat(G))[\overline{K}_{|\frat(G)|}]
\end{equation}  
which immediately implies the following.

\begin{lemma}\label{lem:degree_frat}
Assume that $G$ is noncyclic. Then for every $g\in G$ we have 
\[
\delta_{\Gamma(G)}(g)=\delta_{\Gamma(G/\frat(G))}(g\frat(G))|\frat(G)|.
\] 
In particular $\delta(\Delta(G))=\delta(\Delta(G/\frat(G))|\frat(G)|$.
\end{lemma}

\begin{rem}\label{rem:cyclic}
Let us address the case where $G \cong C_n$. Here $\Delta(G)$ is the graph obtained from $\Delta(G/\frat(G))[K_{|\frat(G)|}]$ by removing all edges between elements of $\frat(G)$. Therefore, $\delta(\Delta(G))=\delta(\Delta(G/\frat(G))|\frat(G)|$, but $|\frat(G)|$ need not divide $\delta_{\Gamma(G)}(g)$ when $g \not\in \frat(G)$ (see Remark~\ref{rem:facts}\ref{it:alpha_beta}).
\end{rem}

We may also deduce the following result on connectedness.
\begin{lemma}\label{lem:connected_frat}
Let $X$ be a subset of $V(\Delta(G))$ of size at least two. Then $X$ is connected if and only if $X\frat(G)$ is connected.
\end{lemma}

\begin{proof} 
Write $F = \frat(G)$. Assume that $X$ is connected and let $x_1f_1, x_2f_2 \in XF$ be distinct. First assume that $x_1 \neq x_2$. Then there is a path $x_1=y_1,y_2,\dots,y_t=x_2$ in $\Delta(G)$ with $y_i\in X$ for $1\leq i\leq t$, and this gives the path  $x_1f_1,y_2,\dots,y_{t-1},x_2f_2$ between $x_1f_1$ and $x_2f_2$. Now assume that $x_1=x_2=x$ and fix $y\in X$ such that $y\neq x$. Then there exists a path $x=y_1,y_2,\dots,y_t=y$ in $\Delta(G)$ with $y_i\in X$ for $1\leq i\leq t$. From this we construct the path $x_1f_1=xf_1,y_2,\dots,y_{t-1},y,y_{t-1},\dots,y_2,xf_2=x_2f_2$ between $x_1f_1$ and $x_2f_2$. Therefore, in both cases, there is a path between $x_1f_1$ and $x_2f_2$, so $XF$ is connected.

Conversely, assume that $XF$ is connected and let $x_1, x_2 \in X$ be distinct. There exists a path $x_1,y_2f_2,\dots,y_{t-1}f_{t-1},x_2$ between $x_1$ and $x_2$ with $y_i\in X$ and this gives the path $x_1,y_2,\dots,y_{t-1},x_2$ in $X$ between $x_1$ and $x_2$, so $X$ is connected.
\end{proof}

We conclude by observing a relationship between group products and graph products.

\begin{lemma}\label{lem:product}
Let $G$ and $H$ be two finite groups. 
\begin{enumerate}
\item If neither $G$ nor $H$ is cyclic, then $\Gamma(G \times H)$ is a subgraph of $\Gamma(G) \times \Gamma(H)$ and $\Delta(G \times H)$ is a subgraph of $\Delta(G) \times \Delta(H)$.
\item If there are no isomorphisms between nontrivial subquotients of $G$ and $H$, then $\Gamma(G) \times \Gamma(H)$ is a subgraph of $\Gamma(G \times H)$ and $\Delta(G) \times \Delta(H)$ is a subgraph of $\Delta(G \times H)$.
\end{enumerate}
\end{lemma}

\begin{proof}
First assume that neither $G$ nor $H$ is cyclic. Assume that $(g_1,h_1)$ and $(g_2,h_2)$ are adjacent in $\Gamma(G \times H)$. Then $\< (g_1,h_1), (g_2,h_2) \> = G \times H$, so $\<g_1,g_2\> = G$ and $\<h_1,h_2\>=H$. Since neither $G$ nor $H$ is cyclic, we conclude that $g_1$ and $g_2$ are adjacent in $\Gamma(G)$ and $h_1$ and $h_2$ are adjacent in $\Gamma(H)$. Therefore, $(g_1,h_1)$ and $(g_2,h_2)$ are adjacent in $\Gamma(G) \times \Gamma(H)$. This prove that $\Gamma(G \times H)$ is a subgraph of $\Gamma(G) \times \Gamma(H)$.

Now assume that there are no isomorphisms between nontrivial subquotients of $G$ and $H$. Assume that $(g_1,h_1)$ and $(g_2,h_2)$ are adjacent in $\Gamma(G) \times \Gamma(H)$. Then $\<g_1,g_2\> = G$ and $\<h_1,h_2\> = H$, so $K = \<(g_1,h_1),(g_2,h_2)\>$ is a subgroup of $G \times H$ that projects onto both $G$ and $H$. By Goursat's Lemma (see \cite[p.75]{lang}, for example), since there are no isomorphisms between nontrivial subquotients of $G$ and $H$, it must be that $K=G \times H$. Therefore, $(g_1,h_1)$ and $(g_2,h_2)$ are adjacent in $\Gamma(G \times H)$, and we conclude that $\Gamma(G) \times \Gamma(H)$ is a subgraph of $\Gamma(G \times H)$.

The claims about $\Delta$ follow from the claims about $\Gamma$ together with the observation that $(g,h)$ is isolated in $\Gamma(G) \times \Gamma(H)$ if and only if $g$ is isolated in $\Gamma(G)$ or $h$ is isolated in $\Gamma(H)$.
\end{proof}

The following is an immediate consequence of Lemma~\ref{lem:product}.
\begin{cor}\label{cor:product}
Let $G$ and $H$ be noncyclic groups of coprime order. Then we have $\Delta(G \times H) = \Delta(G) \times \Delta(H)$.
\end{cor}

\subsection{Nilpotent groups}

Let $G$ be a finite $2$-generated nilpotent group, and write 
\begin{equation} \label{eq:notation}
|G| = \prod_{i=1}^{r} p_i^{a_i} \prod_{j=1}^{s} q_j^{b_j}
\end{equation} 
where $p_1,\dots,p_r,q_1,\dots,q_s$ are the distinct prime divisors of $|G|$ and $p_1,\dots,p_r$ are exactly the prime divisors of $|G|$ for which the Sylow subgroups of $G$ are cyclic. 

Recall that $G=\<g,h\>$ if and only if $G/\frat(G) = \<g\frat(G),h\frat(G)\>$. Here 
\begin{equation} \label{eq:frat}
G/\frat(G) = C_{p_1} \times \cdots \times C_{p_r} \times C_{q_1}^2 \times \cdots \times C_{q_s}^2.
\end{equation}
Therefore,
\begin{equation} \label{eq:gen_frat}
G/\frat(G) = \<(g_1,\dots,g_r,x_1,\dots,x_s), (h_1,\dots,h_r,y_1,\dots,y_s) \>
\end{equation}
if and only if $g_i \neq 1$ or $h_i \neq 1$ for all $1 \leq i \leq r$ and $1 \neq \< x_i \> \neq \< y_i \> \neq 1$ for all $1 \leq i \leq s$. 

\begin{rem} \label{rem:facts}
Let us record some consequences of \eqref{eq:gen_frat}.
\begin{enumerate}[(i)]
\item Let $P_G(2)$ be the probability that a pair of uniformly randomly chosen elements of $G$ generate $G$. Then
\begin{equation*}
P_G(2) = \prod_{i=1}^{r}\left(1-\frac{1}{p_i^2}\right) \prod_{j=1}^{s}\left(1-\frac{1}{q_j^2}\right)\left(1-\frac{1}{q_j}\right).
\end{equation*} \label{it:prob}
\item Since $g \in \Delta(G)$ if and only if $q_1 \dots q_s$ divides $|\frat(G)g|$,
\begin{equation*}
|V(\Delta(G))| = |G| \prod_{j=1}^{s} \left( 1 - \frac{1}{q_j^2} \right),
\end{equation*}
and consequently the proportion of nonisolated vertices in $\Gamma(G)$ is
\begin{equation*}
\frac{|V(\Delta(G)|}{|G|} = \prod_{j=1}^{s} \left( 1 - \frac{1}{q_j^2} \right) \geq \prod_{\text{$p$ prime}} \left( 1 - \frac{1}{p^2} \right) = \frac{6}{\pi^2}.
\end{equation*} \label{it:nonisolated}
\item Let $I \subseteq \{1,\dots,r\}$. Write $\alpha_I$ for the number of elements $g \in G$ such that $\frat(G)g$ has order $\prod_{i \in I} p_i \prod_{j=1}^{s} q_j$ and write $\beta_I$ for the degree in $\Gamma(G)$ of such an element. Then
\begin{equation*}
\alpha_I = |G| \prod_{i \in I} \left( 1 - \frac{1}{p_i} \right) \prod_{i \not\in I} \frac{1}{p_i} \prod_{j=1}^{s} \left( 1 - \frac{1}{q_j^2} \right)
\end{equation*}
and
\begin{equation*}
\beta_I = |G| \prod_{i \not\in I} \left( 1 - \frac{1}{p_i} \right) \prod_{j = 1}^{s} \left( 1 - \frac{1}{q_j} \right) - \epsilon,
\end{equation*}
where $\epsilon$ is $1$ if $G$ is cyclic and $|I|=r$ and is $0$ otherwise. (Note that $\epsilon$ accounts for the fact that we do not consider loops in the generating graph of a cyclic group). \label{it:alpha_beta}
\item Let us record that
\begin{equation*}
\delta(\Delta(G)) = \beta_{\emptyset} = |G| \prod_{i=1}^{r} \left( 1 - \frac{1}{p_i} \right) \prod_{j = 1}^{s} \left( 1 - \frac{1}{q_j} \right)
\end{equation*}
and every vertex $g$ of minimal degree in $\Delta(G)$ satisfies $|g\frat(G)|=q_1 \cdots q_s$. \label{it:delta}
\end{enumerate}
\end{rem}

Remark~\ref{rem:facts}\ref{it:nonisolated} demonstrates that for nilpotent groups $G$ the proportion of nonisolated vertices of $\Gamma(G)$ is at least $6/\pi^2$. In the following example, we show that, even within the class of supersoluble groups, we can find a sequence of groups $(G_d)_d$ for which $|V(\Delta(G_d))|/|G_d| \to 0$.

\begin{ex} \label{ex:isolated}
Let $H = C_2^2$ and let $h_1$, $h_2$, $h_3$ be the nontrivial elements of $H$. Let $p_1 < p_2 < p_3 < \cdots$ be the odd prime numbers. Fix $d \geq 1$. For each $1 \leq i \leq d$, write $N_i = C_{p_i}^3$ and define
\[
G_d = \left(\prod_{i=1}^{d} N_i \right) \rtimes H
\]
where for each $1 \leq i \leq d$ the subgroup $N_i$ is $H$-stable and for $(n_{i1},n_{i2},n_{i3}) \in N_i$ and $h_j \in H$
\[
(n_{i1},n_{i2},n_{i3})^{h_j} = \left\{ \begin{array}{ll} (n_{i1},n_{i2}^{-1},n_{i3}^{-1}) & \text{if $j=1$} \\ (n_{i1}^{-1},n_{i2},n_{i3}^{-1}) & \text{if $j=2$} \\ (n_{i1}^{-1},n_{i2}^{-1},n_{i3}) & \text{if $j=3$.} \end{array} \right.
\]  
By \cite[Proposition~2.2]{luma}, the vertex $(n_{11},\dots,n_{d3};h)$ is nonisolated in $\Gamma(G_d)$ if and only if $h=h_j$ for some $1 \leq j \leq 3$ and $n_{ij} \neq 0$ for all $1 \leq i \leq d$. Therefore,
\[
\frac{|V(\Delta(G_d))|}{|G_d|} = \frac{3 \prod_{i=1}^{d} p_i^2(p_i-1)}{4 \prod_{i=1}^{d} p_i^3} = \frac{3}{4} \prod_{i=1}^{d}\left(1-\frac{1}{p_1}\right),
\]
which tends to zero as $d$ tends to infinity.

Let us note that two vertices $(n_{11},\dots,n_{d3};h_j)$ and $(m_{11},\dots,m_{d3};h_k)$ of $\Delta(G_d)$ are adjacent if $j \neq k$ and for each $1 \leq i \leq d$ we have $n_{dl} \neq m_{dl}$ where $l \not\in \{j,k\}$, so $\Delta(G_d)$ is a regular graph of degree $2 \prod_{i=1}^{d} p_i(p_i-1)^2$.
\end{ex}

We conclude by studying the extent to which $\Gamma(G)$ determines $G$.
\begin{prop} \label{prop:determine}
Let $G$ and $H$ be finite $2$-generated nilpotent groups. Then $\Gamma(G) \cong \Gamma(H)$ if and only if $G/\frat(G) \cong H/\frat(H)$ and $|\frat(G)| = |\frat(H)|$.
\end{prop}

\begin{proof}
If $G/\frat(G) \cong H/\frat(H)$ and $|\frat(G)| = |\frat(H)|$, then \eqref{eq:graph_decomp} (together with Remark~\ref{rem:cyclic}) implies that $\Delta(G) \cong \Delta(H)$ and since $|G|=|H|$ we deduce that $\Gamma(G) \cong \Gamma(H)$. Therefore, it remains to show that from $\Gamma(G)$ we can determine the order of $\frat(G)$ and the isomorphism type of $G/\frat(G)$. Observe that $\Gamma(G)$ has no isolated vertices if and only if $G$ is cyclic. Now assume that $G$ is not cyclic. By \eqref{eq:frat}, we see that it is sufficient to deduce the set $\{ p_1, \dots, p_r\}$ from $\Gamma(G)$ and we can do this by computing the following quotient using Remark~\ref{rem:facts}\ref{it:nonisolated} and~\ref{it:alpha_beta}:
\[
\frac{|\{ v \in \Gamma(G) \mid \delta(v) \neq 0 \}|}{|\{ v \in \Gamma(G) \mid \delta(v) = \delta(\Delta(G)) \}|} = \frac{|V(\Delta(G))|}{\alpha_{\emptyset}} = \frac{|G| \prod_{j=1}^{s}\left( 1-\frac{1}{q_j^2}\right)}{|G| \prod_{i=1}^{r} \frac{1}{p_i} \prod_{j=1}^{s} \left( 1 - \frac{1}{q_j^2} \right)} = p_1 \cdots p_r. \qedhere
\]
\end{proof}

\section{Connectivity} \label{s:conn}

Let $\Gamma$ be a finite graph. Recall that the \emph{(vertex) connectivity} of $\Gamma$, written $\kappa(\Gamma)$, is the least size of a subset $X \subseteq V(\Gamma)$ such that the induced subgraph on $V(\Gamma) \setminus X$ is disconnected (if $\Gamma = K_n$, then we say that $\kappa(\Gamma)=n-1$). It is clear that $\kappa(\Gamma) \leq \delta(\Gamma)$ and we say that $\Gamma$ is \emph{maximally connected} if $\kappa(\Gamma) = \delta(\Gamma)$.

In this section, we will prove Theorem~\ref{thm:conn}, which asserts that $\Delta(G)$ is maximally connected if $G$ is a finite $2$-generated nilpotent group. We begin with a general result.

\begin{lemma}\label{lem:conn_frat}
Let $G$ be a finite $2$-generated group. Then $\kappa(\Delta(G))=\kappa(\Delta(G/\frat(G)))|\frat(G)|$.
\end{lemma}

\begin{proof}
Let us begin by fixing some notation. First write $F=\frat(G)$ and now write $\alpha=\kappa(\Delta(G))$, $\beta=\kappa(\Delta(G/F))$, $\mu=|V(\Delta(G))|$ and $\nu=|V(\Delta(G/F))|$. Notice that $\mu=\nu|F|$. We want to prove that $\alpha =  \beta|F|$.

First assume that $\Delta(G/F))$ is a complete graph. By Lemma~\ref{lem:complete}, there are two possibilities: either $G/F$ has prime order or $G/F \cong C_2^2$.

For now assume that $G/F \cong C_p$, so $\nu=p$ and $\beta=p-1$. Now there exists $t \geq 1$ such that $G = C_{p^t}$ and $|F|=p^{t-1}$. Clearly, $F$ is a disconnected subset of $\Delta(G)$, so $\alpha \leq |G|-|F|=p^t-p^{t-1}$. However, a disconnected subset cannot contain any of $\phi(p^t) = p^t - p^{t-1}$ elements of $G$ that generate $G$, so $\alpha \geq p^t-p^{t-1}$. Therefore, we conclude that $\alpha=p^t-p^{t-1} = \beta|F|$.

Next assume that $G/F \cong C_2^2$, so $\nu=3$ and $\beta=2$. In this case, $G$ is a non-cyclic $2$-generated group of order $2^t$, for some $t \geq 2$, and $|F|=2^{t-2}$. If $g \notin F$, then $gF$ is a disconnected subset of $\Delta(G)$, so $\alpha \leq |V(\Delta(G))|-|F|=2|F|$. Now suppose that $X \subseteq V(\Delta(G))$ is a disconnected subset of size strictly greater than $|F|$. Then there exist $x,y \in X$ such that $\<x,y\>=G$, but then for all $z \in X$ either $\<x,z\>=G$ or $\<y,z\>=G$, so $X$ is connected, which is a contradiction. Therefore, any disconnected subset of $\Delta(G)$ has size at most $F$, so we conclude that $\alpha \geq |V(G)|-|F| =2|F|$. Consequently $\alpha = 2|F| = \beta|F|$.
	
We may now assume that $\Delta(G/F)$ is not a complete graph, or equivalently, that $\nu-\beta \geq 2$. There exists $\Omega \subseteq V(\Delta(G))$ of size $|\Omega|=\nu-\beta$ such that $\{ gF \mid g \in \Omega \}$ is a disconnected subset of $V(\Delta(G/F))$. By Lemma~\ref{lem:connected_frat}, $\Lambda=\{gf \mid g \in \Omega \text{ and } f \in F\}$ is a disconnected subset of $V(\Delta(G))$. Hence 
\[
\alpha \leq \mu-|\Lambda|=\nu|F|-(\nu-\beta)|F|=\beta|F|.
\] 
Let $\Sigma$ be a disconnected subset of $\Delta(V(G))$ of size $\mu-\alpha$. Since $\Sigma F$ is disconnected and $\mu-\alpha$ is the largest size of a disconnected subset of $V(\Delta(G))$, we deduce $\Sigma F=\Sigma$, so in particular $\{ gF \mid g \in \Sigma \}$ is a disconnected subset of $V(\Delta(G/F))$ of size $|\Sigma|/|F|$. Hence 
\[
\beta \leq \nu-\frac{|\Sigma|}{|F|}=\frac{\mu}{|F|}-\frac{\mu-\alpha}{|F|}=\frac{\alpha}{|F|}.
\]
We can therefore conclude that, in all cases, $\alpha = \beta|F|$, as required.
\end{proof}

For the remainder of this section, assume that $G$ is a finite $2$-generated nilpotent group. Let us first consider three special cases.

\begin{lemma}\label{lem:conn_cyclic}
Let $G \cong C_n$. Then $\kappa(\Delta(G)))=\delta(\Delta(G)))=\phi(n)$.
\end{lemma}

\begin{proof}
Since $\delta(\Delta(C_n)) = \delta(1) = \phi(n)$, it remains to show that $\kappa(\Delta(C_n)) \geq \phi(n)$. Now let $X\subseteq V(\Delta(C_n))$. If $|X|>n-\phi(n)$, then there exists $x \in X$ with $\< x \>=G$. Since $x$ is adjacent to all the other vertices of $\Delta(G)$, the subgraph of $\Delta(C_n)$ induced by $X$ is connected. Therefore, $\kappa(\Delta(C_n)) = \phi(n) = \delta(\Delta(C_n))$.
\end{proof}

\begin{lemma}\label{lem:conn_ptimesp}
Let $G \cong C_p^2$ where $p$ is prime. Then $\kappa(\Delta(G)))=\delta(\Delta(G)))=\phi(p^2) = p^2-p$.
\end{lemma}

\begin{proof}
Since $\delta(\Delta(G)) = p^2-p$ (see Remark~\ref{rem:facts}\ref{it:delta}), we aim to show that $\kappa(\Delta(G)) \geq p^2-p$. Let $X$ be a nonempty disconnected subset of $V(\Delta(G))$. Fix $x \in X$. We claim that $X \subseteq \<x\>$. Suppose for a contradiction that $X \not\subseteq \<x\>$. Let $\Omega_x$ be the connected component of the subgraph of $\Delta(G)$ induced by $X$ that contains $x$. If $y \in X \setminus \< x \>$, then $y$ is adjacent to $x$ and $y \in \Omega_x$, so $X \setminus \<x\> \subseteq \Omega_x$. Now fix $y \in X \setminus \< x \>$. Then every vertex in $\langle x\rangle$ is adjacent to $y$, so $X \subseteq \Omega_x$ and $X$ is connected, which is a contradiction. Therefore, $X \subseteq \langle x\rangle$, so $\kappa(G) \geq p^2-1-|X| \geq p^2-p$.
\end{proof}

\begin{lemma}\label{lem:conn_cyclic_ptimesp}
Let $G \cong C_n \times C_p^2$ where $p$ is prime and $n \geq 1$ satisfies $(n,p)=1$. Then $\kappa(\Delta(G))=\delta(\Delta(G))= \phi(|G|)$.
\end{lemma}

\begin{proof}
By Remark~\ref{rem:facts}\ref{it:delta}, $\delta(\Delta(G)) = \phi(|G|)$, and we will show that $\kappa(\Delta(G)) \geq \phi(|G|)$. Let $X$ be a nonempty disconnected subset of $V(\Delta(G))$. Suppose that $|V(\Delta(G))\setminus X| < \phi(|G|) = \phi(n)(p^2-p)$. Since $\phi(n)(p^2-1)> \phi(n)(p^2-p)$, we may fix $(x,a) \in X$, where $C_n = \<x\>$ and $a \in C_p^2$ is nontrivial. Since $|V(\Delta(G))\setminus X| < \phi(n)(p^2-p)$, there exists $(y,b) \in X$ such that $C_n = \<y\>$ and $b \in C_p^2 \setminus \<a\>$. Now every vertex in $\Delta(G)$ is adjacent to either $(x,a)$ or $(y,b)$, so, in particular, $X$ is connected, which is a contradiction. Therefore, $\kappa(G) \geq |V(\Delta(G)) \setminus X| \geq \phi(|G|)$, as desired.
\end{proof}

To complete our proof for general nilpotent groups, we will make use of work in \cite{ww}, where the authors study the connectivity of a direct product of a general graph $\Gamma$ with a complete multipartite graph  $K_{t_1,\dots,t_u}$ under certain restrictions on the parameters $t_1,\dots,t_u$. More precisely, they prove that 
\begin{equation}\label{mulpi}
\kappa(\Gamma\times K_{t_1,\dots,t_u})=\min\left(\kappa(\Gamma)\sum_{i=1}^{u}t_i,\ \delta(\Gamma)\sum_{i=1}^{u-1}t_i\right),
\end{equation} 
if $u \geq 3$ and the sequence $t_1 \leq \dots \leq t_u$ satisfies $\sum_{i=1}^{u-2}t_i \geq t_{u-1}$ and $\sum_{i=1}^{u-1}t_i \geq t_u$.

We are now in the position to prove that $\Delta(G)$ is maximally connected for any $2$-generated nilpotent group $G$. We will adopt the notation from \eqref{eq:notation}.

\begin{proof}[Proof of Theorem~\ref{thm:conn}]
If $G$ is cyclic, then the result is given by Lemma~\ref{lem:conn_cyclic}. Therefore, for the remainder of the proof, we will assume that $G$ is noncyclic. By Lemmas~\ref{lem:degree_frat} and~\ref{lem:conn_frat},
\[ 
\delta(\Delta(G))=\delta(\Delta(G/\frat(G)))|\frat(G)| \text{ \ and  \ } \kappa(\Delta(G))=\kappa(\Delta(G/\frat(G)))|\frat(G)|.
\] 
Therefore, it suffices to show that $\delta(\Delta(G/\frat(G))) = \kappa(\Delta(G/\frat(G)))$, so by replacing $G$ by $G/\frat(G)$, we may assume that $\frat(G)=1$. Consequently, we may write $G= C_n \times C_{q_1}^2 \times \dots \times C_{q_s}^2$, where $n$ is square-free and $(n,q_1\cdots q_s)=1$ (see \eqref{eq:frat}).

We prove our statement by induction on $s$. If $s = 1$, then the result immediately holds by Lemmas~\ref{lem:conn_ptimesp} and~\ref{lem:conn_cyclic_ptimesp}, so we may assume $s \geq 2$. Write $H = C_n \times C_{q_1}^2 \times \dots \times C_{q_{s-1}}^2$ and $K=C_p^2$, where $p=q_s$. By Corollary~\ref{cor:product}, $\Delta(G)= \Delta(H) \times \Delta(K)$. Note that $\Delta(K)=K_{p-1,\dots,p-1}$, a complete multipartite graph with $p+1$ parts of size $p-1$. In particular, $\delta(\Delta(K)) = p^2-p$. By induction, $\kappa(\Delta(H))=\delta(\Delta(H))$, so it follows from \eqref{mulpi} that 
\begin{align*}
\kappa(\Delta(G)) &= \kappa(\Delta(H) \times \Delta(K)) = \delta(\Delta(H))(p^2-p) \\
                  &= \delta(\Delta(H)) \cdot \delta(\Delta(K)) = \delta(\Delta(H) \times \Delta(K)) = \delta(\Delta(G)).
\end{align*}
This completes the proof.
\end{proof}

\begin{rem}\label{rem:edge_conn}
One may also define the \emph{edge connectivity} of $\Gamma$, written $\lambda(\Gamma)$, as the least size of a subset $X \subseteq E(\Gamma)$ such that the subgraph defined by $V(\Gamma)$ and $E(\Gamma) \setminus X$ is disconnected (if $\Gamma=K_1$, then we say that $\lambda(\Gamma)=0$). A result of Whitney \cite{whi} is that $\kappa(\Gamma) \leq \lambda(\Gamma) \leq \delta(\Gamma)$. By \cite{diam}, if $G$ is a $2$-generated finite nilpotent group (or more in general if $G$ is a 2-generated finite group and the derived subgroup of $G$ is nilpotent) then $\diam(\Delta(G))\leq 2$, so it follows from \cite[Theorem 3.3]{hw} that $\lambda(\Delta(G))=\delta(\Delta(G))$ ($\Delta(G)$ is said to be \emph{maximally edge connected}). Our contribution has been to show that, in fact, in this case $\kappa(\Delta(G)) = \lambda(\Delta(G)) = \delta(\Delta(G))$.
\end{rem}

\section{Eulerian and Hamiltonian cycles} \label{s:cycles}

In this section, we prove Theorem~\ref{thm:cycles}. We continue to assume that $G$ is a finite $2$-generated nilpotent group and we adopt the notation from \eqref{eq:notation}. A graph $\Gamma$ is \emph{Hamiltonian} (respectively, \emph{Eulerian}) if it contains a cycle containing every vertex (respectively, edge) of $\Gamma$ exactly once. 

We begin by showing that $\Delta(G)$ is Eulerian unless $G$ is a cyclic group of even order.

\begin{proof}[Proof of Theorem~\ref{thm:cycles}(i)]
Recall that a connected graph $\Gamma$ is Eulerian if and only if all vertices of $\Gamma$ have even degree. First assume that $G$ is a cyclic group of even order. Then the elements of order $|G|$ have degree $|G|-1$, which is odd, so $\Delta(G)$ is not Eulerian. 

For the remainder of the proof we may assume that $G$ is not a cyclic group of even order. Let $g \in \Delta(G)$. If the order of $g\frat(G)$ is $\prod_{i \in I}p_i\prod_{j=1}^{s}q_i$, then, by Remark~\ref{rem:facts}\ref{it:alpha_beta}, the degree of $g$ is 
\[
\delta(g) = \beta_I = |G| \prod_{i \not\in I} \left( 1 - \frac{1}{p_i} \right) \prod_{j = 1}^{s} \left( 1 - \frac{1}{q_j} \right) = \prod_{i \in I} p_i^{a_i} \prod_{i \not\in I} (p_i^{a_i}-p_i^{a_i-1}) \prod_{j=1}^{s} (q_j^{b_j}-q_j^{b_j-1}) - \epsilon,
\]
where $\epsilon$ is $1$ if $G$ is cyclic and $|I|=r$ and $0$ otherwise.

Assume for now that $G$ is noncyclic. In this case, $s \geq 1$, so $q_1^{b_1}-q_1^{b_1-1}$ is even and therefore $\delta(g)$ is even, noting that $\epsilon=0$. Now assume that $G$ is cyclic (necessarily of odd order) but $|I| < r$. Then there exists $i \in I$ such that $(p_i^{a_i}-p_i^{a_i-1})$, which is even, divides $\delta(g)$, noting again that $\epsilon=0$. Therefore, $\delta(g)$ is even. It remains to consider the case where $G$ is cyclic and $|I|=r$. In this case, $\delta(g) = |G|-\epsilon = |G|-1$, which is even. Therefore, in all cases, $\Delta(G)$ is Eulerian.
\end{proof}

We now turn to Hamiltonian cycles. We first consider two special cases.

\begin{lemma}\label{lem:ham_cyclic}
Let $G$ be a finite cyclic group. Then either $\Delta(G)$ is Hamiltonian or $G \cong C_2$ and $\Delta(G) \cong K_2$.
\end{lemma}

\begin{proof}
If $G=\langle g \rangle$ is a cyclic group of order $n > 2$, then $(1,g,g^2,\dots,g^{n-1})$ is an Hamiltonian cycle. The claim is clear if $G \cong C_2$.
\end{proof}

\begin{lemma}\label{lem:ham_p}
Let $G$ be a $2$-generated finite $p$-group. Then either $\Delta(G)$ is Hamiltonian or $G \cong C_2$ and $\Delta(G)\cong K_2$.
\end{lemma}

\begin{proof}
By Lemma~\ref{lem:ham_cyclic}, we may assume that $G$ is not cyclic. Write $|G|=p^n$. The graph $\Delta(G)$ has $|G|(1-\frac{1}{p^2}) = p^n-p^{n-2}$ vertices (see Remark~\ref{rem:facts}\ref{it:nonisolated}), each of them of degree $|G|(1-\frac{1}{p}) = (p^n-p^{n-1})$ (see Remark~\ref{rem:facts}\ref{it:alpha_beta}). A classic theorem of Dirac \cite{dir} states that a graph $\Gamma$ is Hamiltonian if $\delta(\Gamma) \geq |V(\Gamma)|/2$.  Since $p^n-p^{n-1} \geq \frac{1}{2}(p^n-p^{n-2})$, we deduce that $\Delta(G)$ is Hamiltonian.
\end{proof}

As we turn to general nilpotent groups, we need some further graph theoretic preliminaries. Let $\mathcal{H}$ be the set of graphs $\Gamma$ such that $\Gamma$ is Hamiltonian and if $|V(\Gamma)|=2k$ is even then there exists a Hamiltonian cycle of $\Gamma$ denoted $C = (0,\dots,2k - 1)$ with the edges $(i,i+ 1)$ modulo $2k$ and two chords $(r,s)$ and $(u,v)$ where $r$ and $s$ are odd and $u$ and $v$ are even. The following theorem \cite[Theorem~1]{sg} demonstrates the significance of the set $\mathcal{H}$.

\begin{thm}\label{thm:acca}
Let $\Gamma_1$ and $\Gamma_2$ be two Hamiltonian graphs. The graph $\Gamma_1 \times \Gamma_2$ is Hamiltonian if and only if at least one of $\Gamma_1$ and $\Gamma_2$ belongs to $\mathcal{H}$
\end{thm}

\begin{lemma}\label{lem:acca_p}
Let $P$ be a nontrivial $2$-generated finite $p$-group where $p$ is odd. Then $\Delta(P) \in \mathcal{H}$ and $\Delta(C_2 \times P)$ is Hamiltonian. 
\end{lemma}

\begin{proof}
First assume that $P$ is cyclic. By Lemma~\ref{lem:ham_cyclic}, $\Delta(P)$ and $\Delta(C_2 \times P)$ are Hamiltonian (noting that $C_2 \times P$ is also cyclic). Moreover, since $|\Delta(P)|$ is odd, we know that $\Delta(P) \in \mathcal{H}$.

For the remainder of the proof, we will assume that $P$ is noncyclic. Write $|P|=p^n$ and $P = \< a, b\>$. In addition, write $\frat(P)=\{f_1,\dots,f_m\}$ where $m=p^{n-2}$ and $f_1=1$.

For each $1 \leq i \leq m$, define the path
\begin{align*}
H_i = (& b       f_i,\ a      f_i,\ ab      f_i,\ ab^2      f_i,\ \dots,\ ab^{p-1}      f_i, \\ 
       & b^2     f_i,\ a^2    f_i,\ a^2b    f_i,\ a^2b^2    f_i,\ \dots,\ a^2b^{p-1}    f_i, \\
       & \vdots \\
       & b^{p-1} f_i,\ a^{p-1}f_i,\ a^{p-1}bf_i,\ a^{p-1}b^2f_i,\ \dots,\ a^{p-1}b^{p-1}f_i).
\end{align*}
Write $H_i = (h_{i1},\dots,h_{ik})$ where $k=p^2-1$. It is straightforward to see that the concatenation $H = (h_{11},\dots,h_{1k},\dots,h_{m1},\dots,h_{mk})$ is a Hamiltonian cycle in $\Delta(P)$. Notice that $\{ h_{11}, h_{13} \} = \{b,ab\}$ and $\{ h_{12},h_{14} \} = \{a,ab^2\}$ are chords in $H$, so $H$ witnesses the fact that $\Delta(P) \in \mathcal{H}$.

Let $G = C_2 \times P$ where $C_2 = \< x \>$. Define $u_{ij} = x^jh_{ij}$ and $v_{ij} = x^jh_{i(j+3)}$. Then $K = (u_{11},\dots,u_{1k},\dots,u_{m1},\dots,u_{mk},v_{11},\dots,v_{1k},\dots,v_{m1},\dots,v_{mk})$ is a Hamiltonian cycle in $\Delta(G)$, noting that $u_{mk} = h_{mk} = a^{p-1}b^{p-1}f_m$ and $v_{11} = xh_{14} = xab^2f_1$ are adjacent in $\Delta(G)$, and $v_{mk} = h_{m3} = abf_m$ and $u_{11} = xh_{11} = bf_1$ are adjacent in $\Delta(G)$. Therefore, $\Delta(G)$ is Hamiltonian.
\end{proof}

We are now in a position to prove that $\Delta(G)$ is Hamiltonian for any finite $2$-generated nilpotent group other than $C_2$.

\begin{proof}[Proof of Theorem~\ref{thm:cycles}(ii)]
Let $G$ be a finite $2$-generated nilpotent group and assume that $G \not\cong C_2$. Let $t = r+s$, be the number of distinct prime divisors of $|G|$. Then $G$ can be written as $P_1 \times \dots \times P_t$, the direct product of its Sylow subgroups. We will prove by induction on $t$ that $\Delta = \Delta(G)$ is Hamiltonian. If $t=1$, then $G$ is a $p$-group, so the conclusion holds by Lemma~\ref{lem:ham_p}. Now assume that $t > 1$. We may assume that $|P_t|$ is odd, so $\Delta(P_t) \in \mathcal{H}$, by Lemma~\ref{lem:acca_p}. If $t=2$ and $P_1 \cong C_2$, then $\Delta(G) = \Delta(P_1 \times P_2)$ is Hamiltonian by Lemma~\ref{lem:acca_p}. Therefore, we may now assume that $P_1 \times \cdots \times P_{t-1} \not\cong C_2$. Consequently, by the inductive hypothesis, $\Delta(P_1 \times \cdots \times P_{t-1})$ is Hamiltonian and Theorem~\ref{thm:acca} implies that $\Delta^* = \Delta(P_1 \times \cdots \times P_{t-1}) \times \Delta(P_t)$ is Hamiltonian. By Lemma~\ref{lem:product}(i), $\Delta^*$ is a subgraph of $\Delta$, so $\Delta$ is also Hamiltonian. This completes the proof.
\end{proof}

\section{Total domination number} \label{s:tdn}

In this section, we turn to the topic of total domination. A \emph{total dominating set} for a finite graph $\Gamma$ is a set $S$ of vertices of $\Gamma$ such that for all $g \in \Gamma$ there exists $s \in S$ such that $g$ and $s$ are adjacent in $\Gamma$. The \emph{total domination number} $\gamma_t(\Gamma)$ of $\Gamma$ is the smallest size of a total dominating set for $\Gamma$. 

We begin by recording two very straightforward facts.

\begin{lemma} \label{lem:tdn_cyclic}
We have $\gamma_t(\Delta(G)) = 1$ if and only if $G$ is cyclic.
\end{lemma}

\begin{lemma} \label{lem:tdn_frattini}
We have $\gamma_t(\Delta(G)) = \gamma_t(\Delta(G/\frat(G)))$.
\end{lemma}

Let us now establish some graph theoretic results.

\begin{lemma} \label{lem:tdn_product}
Let $\Gamma$ and $\Delta$ be two graphs. Then $\gamma_t(\Gamma \times \Delta) \leq \gamma_t(\Gamma)\gamma_t(\Delta)$.
\end{lemma}

\begin{proof}
If $S \subseteq \Gamma$ is a total dominating set of size $\gamma_t(\Gamma)$ and $T \subseteq \Delta$ is a total dominating set of size $\gamma_t(\Delta)$, then $S \times T$ is a total dominating set for $\Gamma \times \Delta$ of size $\gamma_t(\Gamma)\gamma_t(\Delta)$.
\end{proof}

Let us adopt the notation that the vertex set of the complete graph $K_n$ is $[n] = \{ 1,2,\dots,n \}$. Now observe that for positive integers $a_1,\dots,a_s$, the graph $K_{a_1} \times \cdots \times K_{a_s}$  has vertex set $[a_1] \times \cdots \times [a_s]$ and two sequences $(x_1,\dots,x_s)$ and $(y_1,\dots,y_s)$ are adjacent if and only if they differ in every coordinate.

\begin{lemma} \label{lem:tdn_graph}
Let $\Gamma = K_{a_1} \times \cdots \times K_{a_s}$ where $2 \leq a_1 \leq \cdots \leq a_s$. Then $\gamma_t(\Gamma) \geq s+1$. Moreover, if $a_1 > s$, then $\gamma_t(\Gamma) = s+1$.
\end{lemma}

\begin{proof}
Let $A = \{ a_1, \dots, a_m \} \subseteq \Gamma$ and write $a_i = (a_{i1},\dots,a_{is})$ for each $i \in [m]$. Assume that $m \leq s$. Then $(a_{11},\dots,a_{mm})$ is not adjacent to any element of $A$, so $A$ is not a total dominating set for $\Gamma$. Therefore, $\gamma_t(\Gamma) \geq s+1$.

Now assume that $a_1 > s$. Since $a_i \geq s+1$ for all $i \in [s]$, we may fix $T = \{ (k,\dots,k) \mid k \in [s+1] \}$. Let $(x_1,\dots,x_s) \in \Gamma$ be arbitrary. Now fix $k \in [s+1] \setminus \{x_1,\dots,x_s\}$ and note that $(k,\dots,k)$ is adjacent to $(x_1,\dots,x_s)$. Therefore, $T$ is a total dominating set for $\Gamma$, so $\gamma_t(\Gamma) \leq |T| = s+1$.
\end{proof}

\begin{cor} \label{cor:tdn_graph}
Let $\Gamma = K_{a_1} \times \cdots \times K_{a_s}$ where $2 \leq a_1 \leq \cdots \leq a_s$. Let $t$ be the least nonnegative integer such that $a_i > s-t$ for all $i > t$. Then $\gamma_t(\Gamma) \leq 2^t(s-t+1)$.
\end{cor}

\begin{proof}
Write $\Gamma = K_{a_1} \times \cdots \times K_{a_t} \times \Delta$, where $\Delta = K_{a_{t+1}} \times \cdots \times K_{a_s}$. Noting that $\gamma_t(K_n) = 2$, by combining Lemmas~\ref{lem:tdn_product} and~\ref{lem:tdn_graph}, we obtain $\gamma_t(\Gamma) \leq \gamma_t(K_{a_1})\cdots\gamma_t(K_{a_t})\gamma_t(\Delta) = 2^t(s-t+1)$.  
\end{proof}

We can give a stronger lower bound than the one in Lemma~\ref{lem:tdn_graph}.

\begin{lemma} \label{lem:tdn_graph_lower}
Let $\Gamma = K_{a_1} \times \cdots \times K_{a_s}$ where $2 \leq a_1 \leq \cdots \leq a_s$. Let $t$ be the least nonnegative integer such that $a_i > s-t$ for all $i > t$. Then
\[
\gamma_t(\Gamma) \geq \left\lceil \frac{a_1}{a_1-1} \left\lceil \frac{a_2}{a_2-1} \cdots \left\lceil \frac{a_s}{a_s-1} \right\rceil \right\rceil \right\rceil = \left\lceil \frac{a_1}{a_1-1} \left\lceil \frac{a_2}{a_2-1} \cdots \left\lceil \frac{a_t}{a_t-1} (s-t+1) \right\rceil \right\rceil \right\rceil \geq s+1.
\]
\end{lemma}

\begin{proof}
Let $T = \{ t_1, \dots, t_k \}$ be a total dominating set for $\Gamma$. Let $k_0 = k$ and for each $1 \leq l \leq s$, let $k_l = \left\lfloor k_{l-1} \frac{a_l-1}{a_l} \right\rfloor$. For each $1 \leq l \leq s$, relabelling the elements of $T$ if necessary, we may assume that $t_{il} = t_{jl}$ for all $k_l < i,j \leq k_{l-1}$. Let $g = (t_{1(k_1+1)},t_{2(k_2+1)},\dots,t_{s(k_s+1)})$. Then $g$ is not adjacent to $t_i$ for all $i > k_s$. Therefore, $k_s \geq 1$. Now
\[
k_s \geq 1 \Leftrightarrow 1 \leq \left\lfloor k_{s-1} \frac{a_s-1}{a_s} \right\rfloor \Leftrightarrow 1 \leq k_{s-1} \frac{a_s-1}{a_s} \Leftrightarrow k_{s-1} \geq \frac{a_s}{a_s-1} \Leftrightarrow k_{s-1} \geq \left\lceil \frac{a_s}{a_s-1} \right\rceil.
\]
Continuing in this manner, we obtain
\[
k = k_0 \geq \left\lceil \frac{a_1}{a_1-1} \left\lceil \frac{a_2}{a_2-1} \cdots \left\lceil \frac{a_s}{a_s-1} \right\rceil \right\rceil \right\rceil,
\]
which proves that
\[
\gamma_t(\Gamma) \geq \left\lceil \frac{a_1}{a_1-1} \left\lceil \frac{a_2}{a_2-1} \cdots \left\lceil \frac{a_s}{a_s-1} \right\rceil \right\rceil \right\rceil.
\]

Note that for any positive integers $m$ and $n$ we have $\left\lceil m \cdot \frac{n+1}{n} \right\rceil \geq m+1$ with equality if and only if $n \geq m$. Therefore,
\[
\gamma_t(\Gamma) \geq \left\lceil \frac{a_1}{a_1-1} \left\lceil \frac{a_2}{a_2-1} \cdots \left\lceil \frac{a_s}{a_s-1} \right\rceil \right\rceil \right\rceil = \left\lceil \frac{a_1}{a_1-1} \left\lceil \frac{a_2}{a_2-1} \cdots \left\lceil \frac{a_t}{a_t-1} (s-t+1) \right\rceil \right\rceil \right\rceil \geq s+1. 
\]
This completes the proof.
\end{proof}

We now apply these combinatorial results to study $\gamma_t(\Delta(G))$ when $G$ is a finite $2$-generated nilpotent group. We adopt the notation in \eqref{eq:notation}.

\begin{prop} \label{prop:tdn}
Let $G$ be a finite $2$-generated noncyclic nilpotent group and assume that $q_1 < \cdots < q_s$ are exactly the primes for which $G$ has a noncyclic Sylow subgroup. Then 
\[
\gamma_t(\Delta(G)) = \gamma_t(K_{q_1+1} \times \cdots \times K_{q_s+1}) \geq s+1.
\] 
In particular, if $q_1 \geq s$, then
\[
\gamma_t(\Delta(G)) = s+1.
\]
\end{prop}

\begin{proof}
Observe that 
\[
\<(g_1,\dots,g_r,x_1,\dots,x_s), (h_1,\dots,h_r,y_1,\dots,y_s) \> = G/\frat(G)
\]
if and only if $g_i \neq 1$ or $h_i \neq 1$  for all $1 \leq i \leq r$ and $1 \neq \< x_i \> \neq \< y_i \> \neq 1$ for all $1 \leq i \leq s$. Therefore, 
\[
(h_1,\dots,h_r,y_1,\dots,y_s) \in \Delta(G/\frat(G))
\]
if and only if for all $1 \leq i \leq r$ we have $y_i \neq 1$.

Let $T = \{ t_1,\dots,t_d \} \subseteq \Delta(G/\frat(G))$ where $t_k = (g_{k1},\dots,g_{kr},x_{k1},\dots,x_{ks})$. Then $T$ is a total dominating set for $\Delta(G/\frat(G))$ if and only if for all $(h_1,\dots,h_r,y_1,\dots,y_s) \in \Delta(G/\frat(G))$, there exists $t_k \in T$ such that for all $1 \leq i \leq r$ we have $g_{ki} \neq 1$ and for all $1 \leq i \leq s$ we have $\<x_{ki}\> \neq \<y_i\>$. Therefore, we can fix generators $g_1,\dots,g_r$ for the subgroups $C_{p_1}, \dots, C_{p_r}$ and, without loss of generality, assume that $g_{ki} = g_k$ for all $1 \leq i \leq r$ and $1 \leq k \leq d$. Now, by identifying $[q_i+1]$ with the set of nontrivial cyclic subgroups of $C_{q_i}^2$ for each $1 \leq i \leq s$, we see that the total dominating sets for $\Delta(G/\frat(G))$ correspond exactly to the total dominating sets for $K_{q_1+1} \times \cdots \times K_{q_s+1}$. Therefore, $\gamma_t(\Delta(G)) = \gamma_t(\Delta(G/\frat(G))) = \gamma_t(K_{q_1+1} \times \cdots \times K_{q_s+1})$. In particular, if $q_1 \geq s$, then Corollary~\ref{cor:tdn_graph} implies that $\gamma_t(\Delta(G)) = s+1$.
\end{proof}

Notice that Theorem~\ref{thm:tdn} is an immediate consequence of Lemma~\ref{lem:tdn_cyclic} and Proposition~\ref{prop:tdn}.

\section{Clique and chromatic numbers} \label{s:clique_chromatic}

Let $\Gamma$ be a finite graph. The \emph{clique number} of $\Gamma$, written $\omega(\Gamma)$, is the greatest $k$ for which $K_k$ is a subgraph of $\Gamma$, and the \emph{chromatic number} of $\Gamma$, written $\chi(\Gamma)$, is the least $k$ such that $\Gamma$ admits a proper $k$-colouring. It is clear that $\chi(\Gamma) \geq \omega(\Gamma)$. In this final section, we prove Theorem~\ref{thm:clique_chromatic}.

\begin{proof}[Proof of Theorem~\ref{thm:clique_chromatic}]
If $G$ is not cyclic, then the statement follows from \cite[Theorem 1.1]{luma}. Therefore, assume that $G=\< g \>$ is a cyclic group of order $n=p_1^{a_1}\cdots p_r^{a_r}$, where $r=\pi(n)$. Let $u=\phi(n)$ and let $x_1,\dots,x_u$ be the elements of $G$ with order $n$. Moreover, for $1\leq i\leq r$, let $y_i=g^{p_i}$. It can be easily seen that $\{x_1,\dots,x_u,y_1,\dots,y_r\}$ induces a complete subgraph of $\Gamma(G)$, so $\omega(\Gamma(G))\geq u+r$. Write $B_0 = \{ x_1,\dots,x_u \}$ and for each $1 \leq i \leq r$ write $B_i=\< g^{p_i}\>$ for each $1 \leq i \leq r$. For $1 \leq i \leq r$ write $C_{i} = G \setminus \cup_{0 \leq j \leq i} B_j$, and for $1 \leq i \leq u$, write $C_{r+i} = \{ x_i \}$. We claim that the colouring of $\Gamma(G)$ with colour classes $C_1, \dots, C_{u+r}$ is a proper colouring. To see this, note that if $g,h \in C_i$ are distinct, for some $1 \leq i \leq u+r$, then $i \leq r$ and $|g|$ and $|h|$ both divide $n/p_i$, so $\<g,h\> \neq G$. Therefore, $\chi(\Gamma(G))\leq u+r$, which allows us to conclude that $\omega(\Gamma(G))=\chi(\Gamma(G))=u+r$.
\end{proof}

\end{document}